\documentclass[11pt, reqno]{amsart}
\usepackage{amsthm,amsmath,amsfonts,amssymb,color,mathtools}
\usepackage{tikz-cd}
\usepackage[bookmarks]{hyperref}
\mathtoolsset{showonlyrefs}

\reversemarginpar

\newtheorem{theorem}{Theorem}[section]

\newtheorem{corollary}[theorem]{Corollary}
\newtheorem{lemma}[theorem]{Lemma}
\newtheorem{definition}[theorem]{Definition}

\newtheorem{preremark}[theorem]{Remark}
\newenvironment{remark}{\begin{preremark}\rm}{\medskip \end{preremark}}

\numberwithin{equation}{section}

\newcommand{\norm}[1]{\left\Vert#1\right\Vert}
\newcommand{\abs}[1]{\left\vert#1\right\vert}

\newcommand{\R}{\mathbb R}

\DeclareMathOperator{\Vol}{Vol}

\def\be{\begin{equation}}
\def\ee{\end{equation}}

\begin{document}
\title[Meyers-Serrin on manifolds]{The Meyers-Serrin theorem on Riemannian manifolds:\\ a survey}

\author[Chan]{Chi Hin Chan}
\address{Department of Applied Mathematics, National Yang Ming Chiao Tung University, 1001 Ta Hsueh Road, Hsinchu, Taiwan 30010, ROC}
\email{cchan@math.nctu.edu.tw}

\author[Czubak]{Magdalena Czubak}
\address{Department of Mathematics\\
University of Colorado Boulder\\ Campus Box 395, Boulder, CO, 80309, USA}
\email{czubak@math.colorado.edu}

\begin{abstract}
We revisit the questions of density of smooth functions, and differential forms, in Sobolev spaces on Riemannian manifolds.  We carefully show equivalence of weak covariant derivatives to weak partial derivatives.
\end{abstract}
\maketitle

 \section{Introduction}
 It is now well-known that in the Euclidean case, for an open subset of $\R^n$, the Sobolev spaces defined  using weak derivatives and as a completion of smooth functions are the same.  This is the celebrated 1964 result of Meyers and Serrin \cite{HW}.  In the case of Riemannian manifolds, the books of Aubin \cite{Aubin_book} and Hebey \cite{Hebey_book} define the Sobolev  spaces using the completion of smooth functions in the appropriate norms. Then the books discuss the questions of density of smooth functions versus density of compactly supported smooth functions. This is a natural question to address for open Riemannian manifolds as the generalization of what happens in the case of $\R^n$.  However, these references do not explicitly address the extension of the Meyers-Serrin Theorem to the Riemannian manifolds. To our knowledge, the article that directly addresses this question is the work of Guidetti,  G\"uneysu, and Pallara \cite{GGP}.  There, the authors study Hermitian vector bundles and the questions of density of smooth sections 
in Sobolev spaces defined using linear partial differential operators.   As one application, the main result is applied to the Sobolev spaces on Riemannian manifolds \cite[Cor 3.6]{GGP}.  

Due to the importance of the Sobolev spaces on Riemannian manifolds, we believe it can be beneficial to have a self-contained proof of the Meyers-Serrin Theorem written for functions and for $1$-forms/vector fields that besides being self-contained could also be written using as elementary language as possible to make the proofs more accessible to a PDE community. 

Another article worth mentioning is by Eichhorn  \cite[Prop 1.4]{Eichhorn}, where for one derivative, the equivalence of the Sobolev spaces using weak derivatives, density of smooth functions and compactly supported smooth functions is stated.  The proof gives an outline of a procedure that can be applied in the case of a Riemannian manifold.  The proof is also done for a general differential $q$-form, and can be applied for $q=0$, which gives the case of functions.  In this proof, the first step is to localize the $W^{1,p}$ $k$-form by multiplying it by a Lipschitz cut-off function.  The second step is to use partition of unity, which then further localizes the form to a single chart.   The motivation behind this lies in the fact that smooth approximation of $W^{1,p}$ entities compactly supported in a single coordinate chart could be carried out via standard mollification.  This is due to the fact the coordinate system identifies the chart domain with some region in the Euclidean space. However, this is exactly
the place at which one has to be more careful about how to switch from the original concept of a weak covariant derivative to that of the weak Euclidean partial derivative arising from the coordinate
chart.  A priori, we only have the existence of the weak covariant derivative that satisfies the integration by parts formula. The details were omitted by Eichhorn.  More is shown in \cite{GGP}, but only for one derivative of a function.

%
%

 Therefore, we also present the details of the equivalence of the weak covariant derivatives to Euclidean partial derivatives, which we believe to be the heart of the matter, which again to our knowledge, has been usually omitted in the past literature.

As the goal is to be self-contained, we begin by reviewing some of the definitions.  This article is written to be accessible to anyone with elementary knowledge of differential geometry and analysis, including graduate students.  The goal is to present the material from the point of view of an analyst.
 
The main contribution of this article is the following theorem and its corollary.

\begin{theorem}\label{wu}
Let $k\geq 1$, and $(\phi,\Omega)$ be a chart on $M$. Suppose $u$ is a function on $M$ that has weak covariant derivatives of order $k$ on $M$.  Then $u$ has weak partial derivatives of order $k$ in the Euclidean sense in the image of the chart $\phi(\Omega)$, and the weak covariant derivatives have the same local coordinate representation as in \eqref{generalc},
where we replace the classical derivatives by weak derivatives.
\end{theorem}

\begin{corollary}\label{wa}
Let $q\geq 1$, and $(\phi,\Omega)$ be a chart on $M$. Suppose $\alpha$ is a $q$-form on $M$ that has a weak covariant derivative on $M$.  Then $\alpha$ has a weak partial derivative in the Euclidean sense in the image of the chart $\phi(\Omega)$, and the weak covariant derivative has the same local coordinate representation as in \eqref{generalca},
where we replace the classical derivatives by weak derivatives.
\end{corollary}

\begin{remark}
 We break the proof into cases depending on the value of $k$.  While there is a clear pattern in the proof, the computations become more involved when $k\geq 3$.  For that reason, we explicitly work out the cases $k=1, 2, 3$, and then use the case $k=3$ as the base case for the induction.  
\end{remark}

Once we have the equivalence of the weak covariant derivatives and weak partial derivatives, we can use it to show the Meyers-Serrin Theorem on Riemannian manifolds which follows from the following theorem.

\begin{theorem}[Meyers-Serrin on $M$] \label{HWM}
Let $k\geq 1$, and  $1\leq p<\infty$. Let $\alpha$ be a function or a differential form.  If $\alpha \in W^{k,p}(M)$, there exists a sequence of $\{\alpha_m\}$, such that $\alpha_m$ is smooth and $\alpha_m \to \alpha$ in $W^{k,p}(M).$
\end{theorem}
We give definitions of Sobolev spaces in Section \ref{defSob}.  The main outline of the proof is as follows.  Given $u$ in $W^{k,p}(M)$, where $u$ here can be a function, vector field or a differential form:

\noindent\textbf{Step 1:}. Cover $M$ by charts, and consider the associated subordinate partition of unity.\\
\noindent\textbf{Step 2:}. Use the equivalence of the weak covariant derivatives to weak partial derivatives, Theorem \ref{wu} and Corollary \ref{wa}, to obtain existence of Euclidean weak partial derivatives, and show the function (or the vector field or the differential form) and its Euclidean weak partial derivatives are in $L^p$ in the image of each chart.\\
\noindent\textbf{Step 3:}. Cite the classical Meyers-Serrin Theorem \cite{HW} to show that in each chart we can obtain a sequence converging in $W^{k,p}$, where $W^{k,p}$ refers to the Euclidean Sobolev space in the image of the chart.  \\
\noindent\textbf{Step 4:}. Use partition of unity to obtain a sequence converging in $W^{k,p}(M)$.

We finish by noting that Theorem \ref{wu}, Corollary \ref{wa} and the estimates similar to the ones in the proof of Step 4 in Theorem \ref{HWM} supplement the hidden details in the proof of  \cite[Prop 1.4]{Eichhorn}.

\section*{Acknowledgements}

We thank the referee for the careful reading of the manuscript and for the helpful comments.

Chi Hin Chan is funded in part by a grant from the Ministry of Science and Technology of Taiwan (109-2115-M-009 -009 -MY2). This work was partially completed while Chi Hin Chan was working as a Center Scientist at the National Center for Theoretical Science of Taiwan R.O.C. 

Magdalena Czubak is funded in part by a grant from the Simons Foundation \# 585745.   

\section{Preliminaries}\label{prelim}
We start with a review of the fundamentals of smooth manifolds and Riemannian geometry.  
\subsection{Riemannian geometry setup}
We always sum with respect to repeated indices.   Given a smooth manifold $M^n$, and a chart $(\phi, \Omega)$, which determines coordinates $(x^1, \dots, x^n)$ on $\Omega$, we recall that the partial derivatives of a function $u$ are defined by
\be\label{partial}
\partial_i u(p)=\frac{\partial u}{\partial {x^i}} (p):=\frac{\partial }{\partial {x^i}}(u\circ \phi^{-1})(\phi(p)), \quad p \in \Omega,
\ee
assuming the function $u\circ \phi^{-1}$ is differentiable, so the right hand side makes sense.
If $M$ is a Riemannian manifold $(M, g)$, we can talk about covariant derivatives of $u$. Let $\mathfrak X(M)$ denote all smooth vector fields on $M$. A linear connection $\nabla$ is a bilinear map 
\[
\nabla: \mathfrak X(M) \times \mathfrak X(M) \to \mathfrak X(M)
\]
written as $\nabla_XY$ such that $\nabla_{fX}Y=f\nabla_XY$, and 
\[
\nabla_X(fY)=X(f)Y+f\nabla_X Y,
\]
 for all smooth $f$ on $M$.  Among all linear connections there exists one unique connection, the Levi-Civita connection, that is both compatible with the metric and symmetric, which means
 \begin{align*}
 Xg(Y,Z)=g(\nabla_XY,Z)+g(Y, \nabla_XZ),\\
 \nabla_XY-\nabla_YX=[X,Y],
 \end{align*}
 for all smooth vector fields $X, Y, Z$ and where $[X, Y]$ is the Lie bracket of vector fields $X$ and $Y$.
The first order covariant derivative of $u$, written in coordinates, is
\[
du=\partial_i u dx^i.
\]
We note that to define $du$, it is enough for $M$ to be a smooth manifold.  The existence of the Levi-Civita connection is needed to define the second covariant derivative of $u$.  Namely
\[
\nabla du=\nabla_j \partial_i u \ dx^i\otimes dx^j,
\]
where $\nabla$ refers to the covariant derivative induced by the Levi-Civita connection, and
\be\label{second}
\nabla_j \partial_i u=\partial_j \partial_i u-\Gamma_{ji}^l \partial_l u,
\ee
where we note we sum with respect to $l$, and $\Gamma_{ji}^l$ are Christoffel symbols given by
 \[
\Gamma_{ij}^l=\frac 12 g^{l m}(\partial_i g_{jm}+\partial_j g_{im}-\partial_m g_{ij}).
\]

 By symmetry of the metric, we can see that $\Gamma_{ij}^l=\Gamma_{ji}^l$.

%

In general, the covariant derivative of $u $ of order $k$ can be written in coordinates as follows.
\[
\nabla^k u=\nabla_{i_k}\dots \nabla_{i_1}u dx^{i_1}\otimes\dots \otimes dx^{i_k},
\]
where
\be\label{generalc}
\nabla_{i_k}\dots \nabla_{i_1}u=\partial_{i_k}(\nabla_{i_{k-1}}\dots\nabla_{i_2}\partial_{i_1} u)-\sum_{\mu=1}^{k-1}\Gamma_{i_k i_\mu}^l\nabla_{i_{k-1}}\dots\nabla_{i_{\mu+1}}\nabla_l \nabla_{i_{\mu-1}}\dots\nabla_{i_2}\partial_{i_1} u.
\ee
For a $q$-form $\alpha$, one covariant derivative is
\be\label{generalca}
\nabla_i \alpha_{i_1\dots i_q}=\partial_i \alpha_{i_1\dots i_q}-\sum_{\mu=1}^{q}\Gamma_{i i_\mu}^l \alpha_{i_1\dots i_{\mu-1}li_{\mu+1}\dots i_q}.
\ee

To help the reader who might be new to this notation, we make some comments.  The notation of $\nabla$ with the subscript $i$ can be found in \cite[p. 244 (9.9) and p. 302]{Frankel}, and it is equivalent to the notation that uses ``$/$" or ``$;$" and then the index $i$, see again \cite[p. 244 (9.9) and p. 302]{Frankel} and also \cite[p.98]{Lee_RG}.  Each one has its advantages. 

Next, we note that the second term on the right hand side of formula \eqref{generalc} has two sums: one with respect to $\mu$, and then for each choice of $\mu$, there is a sum with respect to $l$.  We also observe that formula \eqref{generalc} works for all $k\geq 1$.  In the case of $k=1$, there is no sum, and we have just $\partial_{i_1} u$ on the right hand side.  (We also note that it might be helpful to sometimes use the notation that $\nabla_{i_1}u=\partial_{i_1}u$.)  In case of $k=2$, we get exactly \eqref{second}.  For future reference, we apply \eqref{generalc} when $k=3$ to obtain
\be\label{third}
\nabla_{i_3}\nabla_{i_2} \nabla_{i_1}u=\partial_{i_3}(\nabla_{i_2}\partial_{i_1} u)-\Gamma_{i_3 i_2}^l\nabla_{l}\partial_{i_1} u-\Gamma_{i_3 i_1}^l\nabla_{i_2}\partial_{l} u.
\ee
Finally, we remark that the formula \eqref{generalc} is consistent with the definition of a covariant derivative of a general covariant $k-1$-tensor (see e.g., \cite[p.98, Prop. 4.18]{Lee_RG}).  This is no surprise since $\nabla^{k-1}u $ is a covariant $k-1$-tensor.  We can also see this in \eqref{generalca} for $k-1=q$.

%

We use the Jacobi formula that says (see for example \cite[p.44 (3.4.8) and (3.4.9)]{Wald})

\be\label{0Jacobif}
\frac 1g\partial_l g=g^{ab}\partial_l g_{ab},
\ee
where $g$ denotes the determinant of the metric $g$, and that summing over $i$ gives
\be\label{0Christo}
\Gamma^i_{il}=\frac 12g^{ab}\partial_l g_{ab}.
\ee
Then 
\be\label{dg}
\partial_l \sqrt{g}=\sqrt{g} \Gamma^i_{il}, \quad\partial_l g^{-\tfrac 12}=-g^{-\tfrac 12} \Gamma^i_{il}.
\ee
Next, if $\theta$ is a $1$-form, then the adjoint $\nabla^\ast \theta$ can be shown to be
\be\label{useful}
\nabla^\ast \theta=-\nabla^j\theta_{j}=-g^{ji}\nabla_i \theta_j= -\nabla_i \theta^i =-\partial_i\theta^i-\Gamma_{ij}^i\theta^j=-\frac1{\sqrt{g}}\partial_i (\sqrt{g} g^{ij}\theta_j),
\ee
where the last equality follows from $\theta^i=g^{ij}\theta_j$ and \eqref{dg}.

The Riemannian metric $g$, by definition is an inner product on tangent vectors.  However, it also induces an inner product on $k$-tensors, which in coordinates, we can write as
\be\label{metric}
g(S,T)=g^{i_1j_1}\cdots g^{i_kj_k}S_{{i_1}\dots{i_k}}T_{j_1\dots j_k}.
\ee
If $M$ is an orientable Riemannian manifold, we can define a measure on $M$, denoted by $\Vol_M$, and so that $\Vol_M$ has a coordinate representation
\[
\Vol_M=\sqrt{\det{g}} dx^1\wedge \dots dx^n.
\]

We also recall the technical property that manifolds are paracompact:  every cover has a locally finite refinement.  Moreover, given an open cover and any basis for the topology, there exists a countable, locally finite open refinement consisting of elements of the basis.  We can take the basis to be what is called regular balls: charts covering $M$ so that they are homeomorphic to Euclidean balls, and so the closure of each chart is compact and contained in another chart that is also homeomorphic to another (larger) Euclidean ball \cite[Thm 1.15, Prop 1.19]{Lee_smooth}.
\subsection{Definition of weak derivatives and Sobolev spaces}\label{defSob}

We start this section with a lemma in a Euclidean setting.  The lemma facilitates connecting regular weak derivatives with covariant derivatives.  The proof is elementary, so we only sketch the main idea.

\begin{lemma}\label{weakt}
Let $D\subset \R^n$ be an open set, and $k\geq 1$.  The following two statements are equivalent.
\begin{itemize}
\item Let $1\leq j_1, \dots, j_k\leq n$ be fixed. There exists an $L^1_{loc}$ function $v$ such that 
\[
\int_D v\psi dx=(-1)^k\int_D u \partial_{j_1}\cdots \partial_{j_k} \psi dx,
\]
for all $\psi\in C^\infty_c(U)$.
\item There exists an $L^1_{loc}$ $k$-tensor field $\tau$ such that
\be\label{tensor_w}
\int_D \tau_{i_k\dots i_1}\psi_{i_k\dots i_i}  dx=(-1)^k\int_D u \partial_{i_1}\cdots \partial_{i_k} \psi_{i_k\dots i_i} dx,
\ee
for all compactly supported in $D$ $k$-tensor fields $\psi$.
\end{itemize}
\end{lemma}
\begin{proof}
If the first statement is true, given a compactly supported tensor $\psi$ with components $\psi_{i_k\dots i_1}$, we can apply the first statement to each $\psi_{i_k\dots i_1}$, and then sum over all the possible indices.  If the second statement is true, given $1\leq j_1, \dots, j_k\leq n$ and a compactly supported function $\psi$, we can let $\psi_{i_k\dots i_1}=\psi\Pi_{l=1}^k\delta_{i_l j_l}$.
\end{proof}

The above lemma is useful as the most natural way to define weak derivatives in a manifold setting is using, of course, integration by parts, which leads to adjoints of covariant derivatives applied to tensors.  We make this precise now, and state the definitions in the increasing order of difficulty, starting with functions.

\begin{definition}[Weak derivative for a function] \label{weaknablaf}
 Let $u$ be an $L^1_{loc}$ integrable function, then $u$ is weakly differentiable if there exists an $L^1_{loc}$ $1$-form $\tau$ such that
\begin{equation}\label{weakf}
\int_M g(\tau, \theta )\Vol_M=\int_M u\nabla ^*  \theta \Vol_M,
\end{equation}
and the above equality holds for any smooth compactly supported  $1$-form $\theta$.
\end{definition}
\begin{remark}
If $u$ had classical derivatives, then $\tau$ would be equal to $du$.  So the above is a definition for a weak $du$.  Using the metric, we could ``raise" the index to obtain a weak gradient $\nabla u$.  However, working with the $1$-form is convenient for higher order covariant derivatives, as those naturally lead to covariant tensors (instead of contravariant tensors, of which $\nabla u$ is a representative).  Finally, it is an exercise to see that $g(du, du)=g(\nabla u, \nabla u)$, so this choice also does not affect the definition of the Sobolev spaces.
\end{remark}
We now can define a Sobolev space $W^{1,p}(M)$. Let $1\leq p<\infty$, $(M,g)$ be a smooth complete Riemannian manifold of dimension $n<\infty$. Define
\[
W^{1,p}(M)=\{u: M\to \R \ | \ u, du \in L^p(M) \},
\]
with the norm
\[
\norm{\alpha}_{W^{1,p}(M)}=\left\{\int_M \abs{u}^p \Vol_M\right\}^\frac 1p+\left\{\int_M g(du, du)^\frac p2 \Vol_M\right\}^\frac 1p,
\]
where $du$ is understood in the weak sense.  For higher order derivatives we have 

 \begin{definition}[Weak $\nabla^\ell$ for functions] \label{weaknablafl}
 Let $u$ be an $L^1_{loc}$ integrable function, then $u$ has a weak $\nabla^\ell$ derivative, $\ell\geq 1$,  if there exists an $L^1_{loc}$ covariant $\ell$-tensor $\tau$ such that
\begin{equation}\label{weakl}
\int_M g(\tau, \theta )\Vol_M=\int_M u(\nabla ^*)^\ell  \theta \Vol_M,
\end{equation}
and the above equality holds for any smooth compactly supported  $\ell$-tensor $\theta$.
\end{definition}

Then
\[
W^{k,p}(M)=\{ u: M\to \R \ | \  u, \nabla^\ell u \in L^p(M), \ell=1, \dots, k \},
\] 
where $\nabla^iu$, $i=1, \dots, k$ are considered in the weak sense.

We finish this section by defining weak covariant derivatives for a $q$-form $\alpha$.  We are in particular interested in the special case of $q=1$ as it has applications to fluid mechanics on manifolds.

 \begin{definition}[Weak $\nabla^l$ for differential forms] \label{weaknablaa}
Let $l\geq 1$, and let $\alpha$ be an $L^1_{loc}$ integrable $q$-form, then $\alpha$ has a weak $\nabla^l$ derivative,  if there exists an $L^1_{loc}$ covariant $q+l$-tensor $\tau$ such that
\begin{equation}\label{weaka}
\int_M g(\tau, \theta )\Vol_M=\int_M g(\alpha, (\nabla ^*)^l  \theta) \Vol_M,
\end{equation}
and the above equality holds for any smooth compactly supported  $q+l$-tensor $\theta$.
\end{definition}

\section{Equivalence of weak derivatives}\label{equiv}
We start with showing the equivalence holds for functions, first for one derivative then for two derivatives. 
 The reader can choose to read only the parts that are directly of interest without having to deal with the technical details of the more involved cases.
\subsection{Functions: one and two derivatives}
Here we prove Theorem \ref{wu} for the case of one and two covariant derivatives.
\begin{proof}
Let $(\phi, \Omega)$ be a local coordinate chart.

\noindent{\textbf{One derivative.}}
  In order for $u$ to have a weak first order partial derivative in the Euclidean sense, we need to show there exists $\tilde \tau_i$, such that
\be\label{need1}
\int_{\phi(\Omega)} \tilde\tau_i \psi dx=-\int_{\phi(\Omega)} (u\circ \phi^{-1})\partial_i \psi dx,
\ee
holds for all compactly supported functions $\psi$ in $\phi(\Omega).$ Let $\psi$ be given.  Then since $u$ has a weak derivative on $M$, \eqref{weakf} gives that there exists some $1$-form $\tau$ such that
\be\label{k1}
\int_\Omega g^{ij}\tau_i \theta_j \sqrt{g}dx^1\wedge\cdots\wedge dx^n=-\int_\Omega u\nabla^i\theta_i \sqrt{g}dx^1\wedge\cdots\wedge dx^n
\ee
for all $\theta$ compactly supported in $\Omega$.  In particular, we can let $$\theta=\theta_i dx^i, \quad \theta_i=\frac 1{\sqrt{g}}g_{ij}\psi^j\circ \phi,$$ where $\psi^j=\delta^{jl}\psi_l=\psi_j$. 
Then by definition of the integration on manifolds, $\delta^i_j=g^{il}g_{lj}$, and \eqref{useful}, \eqref{k1} is equivalent to
\be\label{0step2eq2}
\begin{split}
\int_{\phi(\Omega)} (\tau_i \circ \phi^{-1})\psi_i dx=-\int_{\phi(\Omega)} u\circ \phi^{-1}\partial_i \psi_i dx,
\end{split}
\ee
 which is exactly \eqref{need1} if we let $\tilde \tau=\tilde \tau_i dx^i$, with 
 \be
 \tilde \tau_i=\tau_i \circ \phi^{-1}.
 \ee

{\textbf{Two derivatives.}}
By the previous case we already have an existence of $\tilde \tau$ as above so that \eqref{need1} holds.  By Lemma \ref{weakt} we then need to show that there exists a $2$-tensor $\tilde A$ such that
\be\label{need2}
\int_{\phi(\Omega)} \tilde A_{ij} \psi_{ij}dx=\int_{\phi(\Omega)} (u\circ \phi^{-1})\partial_j\partial_i \psi_{ij}dx,
\ee
holds for all compactly supported $\psi$ in $\phi(\Omega).$ Let $\psi$ be given. 

By definition of having a weak second covariant derivative of $u$, Definition \ref{weaknablafl}, we have an existence of a $2$-tensor $A$ such that
\be\label{2cov}
\int_M g(A, \theta )\Vol_M=\int_M u\nabla ^*\nabla ^*  \theta \Vol_M,
\ee
for all $2$-tensors $\theta$.   Let $\theta$ be a $2$-tensor such that 
\be\label{theta2}
\theta_{ij}=g_{il}g_{jm}\frac{\psi^{lm}\circ \phi}{\sqrt{g}},
\ee
with $\psi^{lm}=\psi_{lm}$.  
Then by \eqref{metric}
\be\label{g2}
g(A,\theta)=g^{il}g^{jm}A_{ij}\theta_{lm}=A_{ij}\frac{\psi_{ij}\circ \phi}{\sqrt{g}}.
\ee
We claim 
\be\label{2weak}
\tilde A_{ij}=(A_{ij}+\Gamma^l_{ij}\tau_l)\circ \phi^{-1}.
\ee
To see this, consider
\begin{align}
&\int_{\phi(\Omega)} \left((A_{ij}+\Gamma^l_{ij}\tau_l)\circ \phi^{-1}\right)\psi_{ij}dx=\int_{\Omega} (A_{ij}+\Gamma^l_{ij}\tau_l)\psi_{ij}\circ \phi dx^1\wedge\cdots \wedge dx^n\nonumber\\
&\qquad\qquad\qquad= \int_\Omega u\nabla ^*\nabla ^*  \theta \Vol_M+\int_{\Omega} \Gamma^l_{ij}\tau_l(\psi_{ij}\circ \phi) dx^1\wedge\cdots \wedge dx^n,\label{endeq}
\end{align}
 where we used \eqref{g2} and \eqref{2cov}.  To handle the first integral on the right hand side we observe that since in coordinates
 \[
 \nabla^\ast \nabla^\ast \theta=\nabla^i \nabla^j \theta_{ji}, 
 \]
 we can think of $\nabla^j\theta_{ji}$ as a coordinate function of a $1$-form $\alpha=\alpha_i dx^i$ with $\alpha_i=\nabla^j\theta_{ji}$, then since we also have an existence of $1$-weak derivative, we can apply \eqref{weakf} with $\theta=\alpha$ to obtain
 \[
  \int_\Omega u\nabla ^*\nabla ^*  \theta \Vol_M=-\int_\Omega g(\tau,  \alpha) \Vol_M.
 \]
 Then from a computation in coordinates using that $\nabla$ commutes with the metric shows
 \[
 g(\tau,\alpha)=g^{ab}\tau_a \alpha_b= g^{ab}\tau_a\nabla^j\theta_{jb}=\tau_a \nabla_l \frac{\psi^{la}\circ \phi}{\sqrt{g}}.
 \]
 Computing further, using \eqref{dg} and the definition of covariant derivative of a $2$-tensor, we arrive at
 \be
 \begin{split}
 g(\tau,\alpha)&= \frac{\tau_a}{\sqrt{g}} (\partial_l \psi^{la}\circ\phi+\Gamma^l_{lb}\psi^{ba}\circ\phi+\Gamma^a_{lb}\psi^{lb}\circ\phi)-\tau_a \frac{\psi^{la}\circ\phi}{\sqrt{g}}\Gamma_{il}^i\\
 &= \frac{\tau_a}{\sqrt{g}} (\partial_l \psi^{la}\circ\phi+\Gamma^a_{lb}\psi^{lb}\circ\phi).
 \end{split}
 \ee
 Inserting this into \eqref{endeq}, we have
 \begin{align*}
\int_{\phi(\Omega)} \left((A_{ij}+\Gamma^l_{ij}\tau_l)\circ \phi^{-1}\right)\psi_{ij}dx&=- \int_{\phi(\Omega)}  \tilde \tau_a (\partial_l \psi^{la}+\Gamma^a_{lb}\circ\phi^{-1}\psi^{lb}) + \Gamma^l_{ij}\circ\phi^{-1}\tilde\tau_l(\psi_{ij}) dx\\
&=-\int_{\phi(\Omega)}  \tilde \tau_a \partial_l \psi^{la} dx\\
&=\int_{\phi(\Omega)} u\partial_a\partial_l \psi^{la}dx,
\end{align*}
 where we used that  $\tilde \tau$ satisfies \eqref{need1}.  Since $\psi^{la}=\psi_{la}$ in the Euclidean setting, this shows \eqref{need2} holds.
 
 Here, we also show that
 \be\label{obvious}
 \partial_i \tilde \tau_j=\tilde A_{ij},
 \ee
 where $\partial_i$ is taken in the weak sense. By properties of $\tilde A_{ij}$, if $\psi$ is compactly supported in $\phi(\Omega)$, it follows
 \[
 \int \tilde A_{ij}\psi dx=\int u \partial_j\partial_i \psi dx=-\int \tilde \tau_j \partial_i \psi dx,
 \]
 where we used properties of $\tilde \tau_j$.  Then $ \tilde A_{ij}$ satisfies the definition of the weak $\partial_i$ of $\tilde \tau_j$ as needed.  Observe, this also implies that $\partial_i \tau_j$ exists in the Euclidean weak sense.
 
 Then from \eqref{obvious}, \eqref{2weak}, and using that $\tau_l=\partial_l u$ in the weak sense, we obtain
 \be\label{shocker}
 A_{ij}=\partial_i\partial_j u-\Gamma^l_{ij}\partial_l u,
 \ee
 which is exactly the weak form of \eqref{second} as needed.
 \end{proof}
 \subsection{Functions: three derivatives and higher}
 \begin{proof}

 To simplify the notation we make the following simplifications.  We write the volume form in coordinates as $\sqrt{g}dx$, and we drop writing out the composition with either $\phi$ or $\phi^{-1}$. We also change the presentation a bit.  The benefit is that it naturally generalizes to discussing covariant derivatives of $q$-forms for $q\geq 1$.  We proceed by induction and start with $k=3$.  The idea is to first show that the tensor $A_{jk}$ from two covariant derivatives discussed above for functions has a weak partial derivative in the Euclidean sense ($k$ is used here as an index).  After that, using standard properties of weak Euclidean derivatives, we can show that $\partial_i \tilde A_{jk}$ exists and is equal to the weak partial derivatives, $\partial_i \partial_j \partial_k u$.  
 
 Technically, we could start with two derivatives as the base case, but three derivatives is where the computation gets more complicated, so it is useful to see that case as the base case.
  
 From Definition \ref{weaknablafl}, we have an existence of a $3$-tensor $B$ such that
\[
\int_M g(B, \theta )\Vol_M=\int_M u(\nabla ^*)^3  \theta \Vol_M,
\]
for all compactly supported $3$-tensors $\theta$.  We can use this to show that $B$ is the weak covariant derivative of $A$ (c.f. \cite[Thm 1 p.261]{Evans}), so that
  \be\label{3cov}
\int_M g(B, \theta )\Vol_M=\int_M g(A,\nabla^\ast \theta) \Vol_M,
\ee
for all compactly supported $3$-tensors $\theta$. To show $\partial_i A_{jk}$ exists,  we note we can use the definition of the integration on manifolds,  to show we can find $\tilde \tau$ such that
 \be\label{need3}
 \int_\Omega \tilde \tau_{ijk}\psi_{ijk} dx=-\int_\Omega A_{jk}\partial_i \psi^{ijk} dx,
 \ee
  where $\psi^{ijk}$ is compactly supported in $\phi(\Omega)$ (and where again we drop the compositions with $\phi$ and $\phi^{-1}$).
   
 We define $\tilde \tau$ to be
 \be\label{3def}
 \tilde \tau_{ijk}=B_{ijk}+\Gamma^l_{ij}A_{lk}+\Gamma^l_{ik}A_{jl}.
 \ee
  
 Let $\psi$ be a compactly supported $3$-tensor in $\phi(\Omega)$, then with $\theta_{ijk}=g_{ia}g_{jb}g_{kc}\frac{\psi^{abc}}{\sqrt{g}},$ we have
 \be\label{left3}
 g(B,\theta){\sqrt{g}}=B_{ijk}{\psi^{ijk}}.
 \ee
 Next, we can show that for a $3$-tensor $X$
 \be\label{3tensordiv}
\nabla^iX_{ijk}=\nabla_iX^i_{jk}=\frac{1}{\sqrt{g}}\partial_i(\sqrt{g}X^i_{jk})-\Gamma^l_{i j}X^i_{lk }-\Gamma^l_{i k}X^i_{jl},
\ee
 which with $X=\theta$ as above gives
 \be\nonumber
 \begin{split}
\nabla^i\theta_{ijk}&=\frac{1}{\sqrt{g}}\partial_i(g_{jb}g_{kc}{\psi^{ibc}})-\Gamma^l_{i j}g_{lb}g_{kc}\frac{\psi^{ibc}}{\sqrt{g}}-\Gamma^l_{i k}g_{jb}g_{lc}\frac{\psi^{ibc}}{\sqrt{g}}\\
&=\frac{1}{\sqrt{g}}g_{jb}g_{kc}\partial_i\psi^{ibc}+\frac{1}{\sqrt{g}}\partial_i(g_{jb}g_{kc}){\psi^{ibc}}-\Gamma^l_{i j}g_{lb}g_{kc}\frac{\psi^{ibc}}{\sqrt{g}}-\Gamma^l_{i k}g_{jb}g_{lc}\frac{\psi^{ibc}}{\sqrt{g}}\\
\end{split}
\ee
 Then
 \[
 g(A, \nabla^\ast \theta)\sqrt{g}=-\left(A_{jk}\partial_i \psi^{ijk} +g^{jj'}g^{kk'}A_{jk}(\partial_i(g_{j'b}g_{k'c})-\Gamma^l_{i j'}g_{lb}g_{k'c}-\Gamma^l_{i k'}g_{j'b}g_{lc}){\psi^{ibc}}\right),
 \]
 and applying 
 \[
 \nabla_kg_{jh}=\partial_k g_{jh} -\Gamma^l_{kj}g_{lh} -\Gamma^l_{kh}g_{jl}=0,
\]
 twice gives
 \begin{align*}
 g(A, \nabla^\ast \theta)\sqrt{g}&=-\left(A_{jk}\partial_i \psi^{ijk} +g^{jj'}g^{kk'}A_{jk}(\Gamma^l_{i b}g_{j'l}g_{k'c}+\Gamma^l_{i c}g_{j'b}g_{k'l}){\psi^{ibc}}\right)\\
 &=-\left(A_{jk}\partial_i \psi^{ijk} +A_{lk}\psi^{ijk}\Gamma^l_{i j}+A_{jl}\psi^{ijk}\Gamma^l_{i k}\right).
 \end{align*}
 When we plug this and \eqref{left3} into \eqref{3cov} we obtain \eqref{need3} with \eqref{3def} as needed.  Note, since $\tilde \tau_{ijk}$ represents weak partial derivative of $A_{jk}$ we also have
 \be\label{obvious2b}
 B_{ijk}=\partial_i A_{jk}-\Gamma^l_{ij}A_{lk}-\Gamma^l_{ik}A_{jl},
 \ee
 which is the weak version of \eqref{generalc}.  To see that we have weak partial derivatives up to order three, we apply weak partial derivative to \eqref{2weak} since all the needed partials exist (using here again properties of the weak partial derivatives, e.g., \cite[Thm 1 p.261]{Evans}).  
 
 \subsubsection{Inductive step} Taking the above case as the base case, we proceed by induction.  We claim that in general the weak partial derivative of the $k-1$-covariant derivative of $u$, $\partial_{a_1}B_{a_2\dots a_{k}}$ is
  \be\label{weakk}
 \tilde \tau_{a_1\dots a_{k}}=B_{a_1\dots a_{k}}+\sum_{\mu=2}^{k}\Gamma^i_{a_1a_\mu}B_{a_2\dots a_{\mu-1} i a_{\mu+1}\dots a_{k} }.
 \ee
 
  Given a compactly supported $k$ tensor $\psi$ in $(\phi(\Omega))$, we then let
 \be\label{0thetak}
 \theta_{b_1\dots b_{k}}=\frac{1}{\sqrt{g}}g_{b_1c_1}\dots g_{b_{k}c_{k}}\psi^{c_1\dots c_{k}},
 \ee
 and we can compute as before to show that the analog of \eqref{3cov} implies the analog of \eqref{need3} (we omit the details).  After that, we claim that in general a $k$-order weak partial derivative is heauristically given by
  \be\label{ih2}
   \tilde B_{a_1\dots a_k}=B_{a_1\dots a_k}+\sum_{m=1}^{k-1} \ \sum_{1\leq b_1,\dots, b_{m}\leq n } f_{b_1\dots b_m} B_{b_1\dots b_m} ,
  \ee

  where $f_{b_1\dots b_m}$ denotes some smooth function which depends on the Christoffel symbols and their derivatives.  
  
%
%
%
Note \eqref{ih2} holds for $k=3$, since by \eqref{2weak}, \eqref{obvious2b} and \eqref{obvious} we have
\be
\partial_i\tilde A_{jk}=\partial_i(A_{jk}+\Gamma^l_{jk}\tau_l)=B_{ijk}+\Gamma^l_{ij}A_{lk}+\Gamma^l_{ik}A_{jl}+\Gamma^l_{jk}A_{il}+\partial_i \Gamma^l_{jk}\tau_l.
\ee
We can then show \eqref{ih2} holds for $k+1$ given \eqref{ih2} and \eqref{weakk} for $k+1$ by applying $\partial_i$ to \eqref{ih2}.
 \end{proof}
 
 \subsection{Covariant derivatives of differential forms}
We now give a formula for a weak partial derivative of a $q$-differential form and leave out the computations needed to verify it.  Given a $q$-form $\alpha$, the weak derivative $\partial_i$ of $\alpha_{a_1\dots a_{q}}$,
 is
 \[
 \tilde \tau_{ia_1\dots a_{q}}=\tau_{ia_1\dots a_{q}}+\sum_{\mu=1}^{q}\Gamma^h_{ia_\mu}\alpha_{a_1\dots a_{\mu-1} h a_{\mu+1}\dots a_{q} },
 \]
  where the existence of $\tau_{ia_1\dots a_{q}}$ is due to having a weak covariant derivative of $\alpha$.  For example, the weak partial derivative of a $1$-form $\alpha$ can be shown to be  (c.f. \eqref{2weak})
 \be
\tilde \tau_{ij}=(\tau_{ij}+\Gamma_{ij}^l \alpha_l)\circ \phi^{-1}.
\ee

\section{Meyers-Serrin Theorem on manifolds}\label{HW}
Having shown the equivalence of the weak covariant derivatives to weak partial derivatives, we can now show the Meyers-Serrin Theorem, Theorem \ref{HWM}, on Riemannian manifolds.  
We provide the details of the outline given in the introduction.  As the proof is the simplest for $W^{1,p}(M)$ for functions, and that is also the one that might be useful the most to the readers, we present its details separately.  The other cases are presented in a separate section.
\subsection{Proof of Theorem \ref{HWM}: Meyers-Serrin Theorem for $W^{1,p}(M)$ for functions}

\noindent\textbf{Step 1:}. Let $u\in W^{1,p}(M)$, and let $\{(\phi_m, \Omega_m)\}$ be charts covering $M$.  Then, by smoothness of the metric, we can show that for each chart $(\phi_m, \Omega_m)$ there exist positive constants $c_m, C_m$ such  that we have
\be\label{step1eq}
c_m\abs{v}_0^2\leq g(v,v) \leq C_m \abs{v}^2_0, \quad c_m \leq \sqrt{g}\leq C_m,\quad \abs{\Gamma_{ij}^l}\leq C_m, 1\leq i, j, l\leq n,
\ee
where $\abs{v}_0$ is the Euclidean norm of $v$, when $v$ is written in coordinates in a given chart.  Note, we can suppose $C_m\geq 1$.

Let $\{\eta_m\}$ be a partition of unity subordinate to $\{(\phi_m, \Omega_m)\}$ . \\

\noindent\textbf{Step 2:}.
Fix a chart $(\phi, \Omega)$ and a corresponding member of the partition of unity, and denote it by $\eta$. Consider the pullback of $\eta u$ to $\phi(\Omega)$, and let $c_m=c$ and $C_m=C$ in \eqref{step1eq}.  Then 
\be\nonumber
\begin{split}
&\norm{\phi^{-1\ast}(\eta u )}_{L^p(\phi(\Omega))}^p=\int_{\phi(\Omega)}\abs{\phi^{-1\ast}(\eta u)}^p_0 dx\\
&\qquad\qquad \leq\frac{1}{c^{\tfrac p2+1}}\int_{\phi(\Omega)}  \abs{\phi^{-1\ast}(\eta u)}_{g\circ \phi^{-1}}^p\sqrt{g\circ \phi^{-1}}dx\leq \frac{1}{c^{\tfrac p2+1}}\norm{\eta u}_{L^p(M)}^p<\infty.
\end{split}
\ee
This shows $\phi^{-1\ast}(\eta u)$ is in $L^p$  on $\phi(\Omega)\subset \R^n$.  From the last section we know we can show that $\eta u$ has a Euclidean weak partial derivative in $\phi(\Omega)$. We now show that partial derivative is in $L^p(\phi(\Omega))$.  We have
\begin{align*}
\norm{\tilde\tau_{i}}_{L^p(\phi(\Omega))}\leq\left\{\int_{\phi(\Omega)}\abs{\phi^\ast\tau}_0^p dx\right\}^\frac1p
& \leq \frac{1}{c^{\tfrac 12+\tfrac1p}} \left\{\int_{\phi(\Omega)}\abs{\phi^\ast\tau}_{g\circ \phi^{-1}}^p\sqrt{g\circ \phi^{-1}} dx\right \}^\frac 1p\\
&\leq \frac{\norm{\tau}_{L^p(M)}}{c^{\tfrac 12+\tfrac 1p}} <\infty.
\end{align*}

    \noindent\textbf{Step 3:} Given that  $\phi^{-1\ast}(\eta u)$ is in $W^{1,p}_0(\phi(\Omega))$ we can approximate it by a compactly supported sequence $\{\tilde u_k\}$.  We can do this in each chart $(\phi, \Omega)$.  If the charts are indexed with an index $m$, then the corresponding sequence can be denoted by $\{\tilde u_{m,k}\}$. Let $\{\tilde u_{m,k}\}$ be chosen so that 
  \be\label{0choosem}
  \norm{\phi_m^{-1\ast}(\eta_mu)-\tilde u_{m,k}}_{W^{1,p}(\phi_m(\Omega_m))}\leq \frac{1}{2^{k}2^m}\frac{1}{C^{\tfrac 12+\tfrac 1p}_m},
  \ee
  We can then form the desired sequence, $\{u_k\}$, by letting for $x\in M$, 
 \be
 u_k(x)=\sum_{m=1}^{\infty} u_{m,k}(x):=\sum_{m=1}^{\infty} \phi^\ast_m \tilde u_{m,k}(x).
 \ee
 Note, by the paracompactness, as reviewed in Section \ref{prelim}, the sum is finite  for each $x$.
 
   \noindent\textbf{Step 4:}. We now show this sequence converges to $u$ in $L^p(M)$.  
  By construction
  \begin{align*}
  \norm{u-u_k}_{L^p(M)}&=\norm{\sum_{m=1}^{\infty}\eta_m u- \sum_{m=1}^{\infty} \phi^\ast_m \tilde u_{m,k}   }_{L^p(M)}\\
&\leq\sum_{m=1}^{\infty}\left\{ \int_{\phi_m(\Omega_m)}\abs{\phi^{-1\ast}_m(\eta_mu)-   \tilde u_{m,k}}^p_{g\circ\phi_m^{-1}}\sqrt{g\circ \phi_m^{-1}} dx\right\}^\frac 1p\\
&\leq\sum_{m=1}^{\infty}C^{\tfrac 12+\tfrac 1p}_m \left\{ \int_{\phi_m(\Omega_m)}\abs{\phi_m^{-1\ast}(\eta_mu)-   \tilde u_{m,k}}^p_0dx\right\}^\frac 1p\\
&\leq\sum_{m=1}^{\infty}C^{\tfrac 12+\tfrac 1p}_m \frac{1}{2^{k}2^m}\frac{1}{C^{\tfrac 12 +\tfrac 1p}_m}=\frac 1{2^{k}}.
  \end{align*}

%
Next we show that $\{d u_k\} $ converges to $d u=\tau$ in $L^p(M)$.
\begin{align*}
\norm{d u -d  u_{k}}_{L^p(M)}
&\leq\sum_{m=1}^{\infty}\left\{\int_{\phi_m(\Omega_m)}\abs{\phi^{-1\ast}_md(\eta_m u)-   \phi^{-1\ast}_md(\phi^\ast_m \tilde u_{m,k})}^p_{g\circ\phi_m^{-1}}\sqrt{g\circ \phi_m^{-1}} dx\right\}^\frac 1p\\
&\leq\sum_{m=1}^{\infty}C^{\tfrac 12+\tfrac1p}_m
\left\{\int_{\phi_m(\Omega_m)}\abs{\tilde \tau- d \tilde u_{m,k}}_0^p dx \right\}^\frac 1p.
\end{align*}
Then \eqref{0choosem} implies convergence in $W^{1,p}(M)$ as needed.
\subsection{Proof of Theorem \ref{HWM}: Meyers-Serrin for $W^{2,p}(M)$ for functions}
Here Step 1 remains exactly the same as in the above proof.  Step 2 gets supplemented with showing that the weak second order partial derivative of $u$ is in $L^p(\phi(\Omega))$.  We use \eqref{2weak} and estimate
\begin{align*}
\norm{\partial_i \partial_j u}_{L^p(\phi_m(\Omega))}&\leq\norm{\nabla du}_{L^p(\phi_m(\Omega))}+C_m\norm{\tau}_{L^p(\phi_m(\Omega))}\\
&\leq \frac{1}{c_m^{\frac 12+\frac 1p}}(\norm{\nabla du}_{L^p(M)}+C_m\norm{du}_{L^p(M)})<\infty.
\end{align*}
 In Step 3, we choose the sequence $\{\tilde u_{m,k}\}$ so that 
  \be\label{0choosem2}
  \norm{\phi^{-1\ast}_m(\eta_mu)-\tilde u_{m,k}}_{W^{2,p}(\phi_m(\Omega_m))}\leq \frac{1}{2^{k}2^m}\frac{1}{C^{\frac 32+\tfrac 1p}_m}.
  \ee
Finally, Step 4 gets supplemented by showing that $\{\nabla d u_k\} $ converges to $\nabla d u$ in $L^p(M)$.
\begin{align*}
\norm{\nabla d u -\nabla d  u_{k}}_{L^p(M)}&=\norm{\nabla d(\sum_{m=1}^{\infty}\eta_m u)- \nabla d(\sum_{m=1}^{\infty} \phi^\ast_m \tilde u_{m,k} )  }_{L^p(M)}\\
&\leq\sum_{m=1}^{\infty} \norm{\nabla d(\eta_m u)-   \nabla d(\phi^\ast_m \tilde u_{m,k})   }_{L^p(M)}.
\end{align*}
Since
\begin{align*}
\nabla_i \partial_j(\eta_m u)-   \nabla_i \partial_j(\phi^\ast_m \tilde u_{m,k}) =\partial_i\partial_j (\eta_m u)-\partial_i\partial_j(\phi^\ast_m \tilde u_{m,k})-\Gamma_{ij}^l\partial_l(\eta_m u)+\Gamma_{ij}^l\partial_l(\phi^\ast_m \tilde u_{m,k}),
\end{align*}
 we have 
\begin{align*}
\norm{\nabla d u -\nabla d  u_{k}}_{L^p(M)}\leq\sum_{m=1}^{\infty}C_m^{\frac 32+\frac 1p} \norm{\phi^{-1\ast}_m(\eta_m u)-   \tilde u_{m,k}  }_{W^{2,p}(\phi_m(\Omega_m))},
\end{align*}
and the convergence follows from \eqref{0choosem2}.  Higher order derivatives are similar.


\begin{thebibliography}{1}

 \bibitem{Aubin_book}
Thierry Aubin.
\newblock {\em Nonlinear analysis on manifolds. {M}onge-{A}mp\`ere equations},
  volume 252 of {\em Grundlehren der mathematischen Wissenschaften [Fundamental
  Principles of Mathematical Sciences]}.
\newblock Springer-Verlag, New York, 1982.

\bibitem{Eichhorn}
J\"{u}rgen Eichhorn.
\newblock Elliptic differential operators on noncompact manifolds.
\newblock In {\em Seminar {A}nalysis of the {K}arl-{W}eierstrass-{I}nstitute of
  {M}athematics, 1986/87 ({B}erlin, 1986/87)}, volume 106 of {\em Teubner-Texte
  Math.}, pages 4--169. Teubner, Leipzig, 1988.

\bibitem{Evans}
Lawrence~C. Evans.
\newblock {\em Partial differential equations}, volume~19 of {\em Graduate
  Studies in Mathematics}.
\newblock American Mathematical Society, Providence, RI, second edition, 2010.

\bibitem{Frankel}
Theodore Frankel.
\newblock {\em The geometry of physics. An introduction}, 
 \newblock Cambridge University Press, Cambridge, third edition, 2012.


\bibitem{GGP}
Davide Guidetti, Batu G\"{u}neysu, and Diego Pallara.
\newblock {$L^1$}-elliptic regularity and {$H=W$} on the whole {$L^p$}-scale on
  arbitrary manifolds.
\newblock {\em Ann. Acad. Sci. Fenn. Math.}, 42(1):497--521, 2017.

\bibitem{Hebey_book}
Emmanuel Hebey.
\newblock {\em Nonlinear analysis on manifolds: {S}obolev spaces and
  inequalities}, volume~5 of {\em Courant Lecture Notes in Mathematics}.
\newblock New York University, Courant Institute of Mathematical Sciences, New
  York; American Mathematical Society, Providence, RI, 1999.

\bibitem{Lee_smooth}
John~M. Lee.
\newblock {\em Introduction to smooth manifolds}, volume 218 of {\em Graduate
  Texts in Mathematics}.
\newblock Springer, New York, second edition, 2013.

\bibitem{Lee_RG}
John~M. Lee.
\newblock {\em Introduction to {R}iemannian manifolds}, volume 176 of {\em
  Graduate Texts in Mathematics}.
\newblock Springer, Cham, 2018.
\newblock Second edition of [ MR1468735].

\bibitem{HW}
Norman~G. Meyers and James Serrin.
\newblock {$H=W$}.
\newblock {\em Proc. Nat. Acad. Sci. U.S.A.}, 51:1055--1056, 1964.

\bibitem{Wald}
Robert~M. Wald.
\newblock {\em General relativity}.
\newblock University of Chicago Press, Chicago, IL, 1984.


\end{thebibliography}
\end{document}